\documentclass[10pt,a4paper]{amsart}
\usepackage[english]{babel}
\usepackage{amssymb,xypic,amscd}
\CompileMatrices

\newtheorem{defn}{Definition}

\newtheorem{cor}[defn]{Corollary}
\newtheorem{lem}[defn]{Lemma}
\newtheorem{prop}[defn]{Proposition}

\theoremstyle{remark}
\newtheorem{rem}[defn]{Remark}
\theoremstyle{remark}
\newtheorem{exam}{Example}

\numberwithin{equation}{section} \numberwithin{defn}{section}

\newcommand\aut{\operatorname{Aut}}
\newcommand\ed{\operatorname{End}}

\newcommand\rk{\operatorname{rk}}

\newcommand\Ker{\operatorname{Ker}}

\newcommand\Det{\operatorname{det}}

\newcommand\tr{\operatorname{tr}}

\newcommand\limpl[1]{\underset{#1}\varprojlim\,}

\def\mod #1/#2{\kern.06em{\raise1.2pt\hbox{$#1$}}/
      {\raise-1.2pt\hbox{$#2$}}}

\newcommand\B{{\mathcal B}}

\renewcommand\tilde{\widetilde}

\renewcommand\lim{\limpl{A\in\B}}
\newcommand\beq{
      \setcounter{equation}{\value{defn}}\addtocounter{defn}1
      \begin{equation}}

\begin{document}

\title[G-Drazin inverses of finite potent endomorphisms and matrices]{On G-Drazin inverses of \\  Finite Potent Endomorphisms \\ and Arbitrary Square Matrices}
\author{Fernando Pablos Romo }

\address{Departamento de
Matem\'aticas and Instituto Universitario de F\'{\i}sica Fundamental y Matem\'aticas, Universidad de Salamanca, Plaza de la Merced 1-4,
37008 Salamanca, Espa\~na} \email{fpablos@usal.es}
\thanks{
\\ This work is partially supported  by the
Spanish Government research projects nos. MTM2015-66760-P and PGC2018-099599-B-I00 and the Regional Government of Castile and Leon research project no. J416/463AC03.}

\maketitle

\begin{abstract} The aim of this work is to extend to finite potent endomorphisms the notion of G-Drazin inverse of a finite square matrix. Accordingly, we determine the structure and the properties of a G-Drazin inverse of a finite potent endomorphism and, as an application, we offer an algorithm to compute the explicit expression of all G-Drazin inverses of a finite square matrix.
\end{abstract}

\setcounter{tocdepth}1

\tableofcontents

\section{Introduction}
 
For an arbitrary $(n\times n)$-matrix $A$ with entries in the complex numbers, the index of $A$, $i(A) \geq 0$, is the smallest integer such that $\rk (A^{i(A)}) = \rk (A^{i(A)+1})$. Given $A \in {\text{Mat}}_{n\times n} ({\mathbb C})$ with $i(A) = r$, H. Wang and X. Liu introduced in \cite{WL} the notion of ``G-Drazin inverse'' of $A$ as a solution $X$ of the system \begin{equation} \label{eq:G-D} \begin{aligned} AXA &= A\, ; \\ XA^{r+1}&= A^r \, ;\\ A^{r+1} X &= A^r\, ,\end{aligned}\end{equation}  \noindent where $X$ is a $(n\times n)$-matrix with entries in $\mathbb C$.

Recently,  C. Coll, M.  Lattanzi and N.  Thome have proved in \cite{CLT} that a matrix $X  \in {\text{Mat}}_{n\times n} ({\mathbb C})$ is a solution of the system (\ref{eq:G-D}) if and only if $X$ satisfies that  \begin{equation} \label{eq:G-D-2} \begin{aligned} AXA &= A\, ; \\ XA^{r}&= A^rX\, .\end{aligned} \end{equation}

On the other hand, if $k$ is a field, $V$ is an arbitrary vector
space over $k$ and $\varphi$ is an endomorphism of $V$,
according to \cite{Ta} we say that $\varphi$ is ``finite-potent"
if $\varphi^n V$ is finite dimensional for some $n$.

During recent years, the author has extended the notions of Drazin inverse, Core-Moore-Penrose inverse and Drazin-Moore-Penrose inverses of finite square matrices to finite potent endomorphisms, and has offered several properties of these extensions (\cite{Pa-CN}, \cite{Pa-DMP} and \cite{Pa-Dr}). In particular, all the results obtained for finite potent endomorphisms are also valid for finite square matrices.

	The aim of this work is to extend to finite potent endomorphisms the notion of G-Drazin inverse of a finite square matrix. Indeed, we determine the structure of a G-Drazin inverse of a finite potent endomorphism and, in particular, 
we offer the explicit expression of all G-Drazin inverses of a finite square matrix.

The paper is organized as follows. In Section \ref{s:pre} we briefly recall the basic definitions of this work: the definition of finite potent endomorphisms with the decomposition of the vector space given by M. Argerami, F. Szechtman and R. Tifenbach in \cite{AST}; the Jordan bases of nilpotent endomorphisms of infinite-dimensional vector spaces; the Drazin inverse of an $(n\times n)$-matrix and a finite potent endomorphism; the core-nilpotent decomposition of a finite potent endomorphism; and the basic properties of the G-Drazin inverses of an square matrix.

Section \ref{s:GDrazin-Inv-FP} is devoted to proving the existence of G-Drazin inverses of finite potent endomorphism (Proposition \ref{p:char-G-D-fp}) and to offering the explicit structure of these linear maps on arbitrary $k$-vector spaces (Lemma \ref{l:char-gen-f-p-basi-r-345} and Corollary \ref{c:char-gen-f-p-basi-r-345}). Moreover, an explicit example of G-Drazin inveres of a finite potent endomorphism is given.

Finally, the goal of Section \ref{s:G-Dra-matri-3455} is to apply the results of Section \ref{s:GDrazin-Inv-FP} to study the set $A\{GD\}$ of G-Drazin inverses of a square matrix $A \in \text{Mat}_{n\times n} (k)$, where $k$ is an arbitrary field. Accordingly,  if the index of $A$ is r and $\rk (A^i)$ is the rank of $A^i$, we can  write $\nu_i  (A) = n - \rk (A^i)$ for all $i\in \{1, \dots, r\}$ and we check that there exists a bijection $$ k^{\nu_1 (A) \cdot \nu_r (A)} \times  k^{[\nu_r (A) - \nu_1 (A)]} \times k^{[\nu_1(A) - 1] [\nu_r (A) - \nu_1 (A)]} \overset {\sim} \longrightarrow A\{GD\}\, ,$$\noindent from where we determine the explicit expression of all G-Drazin inverses of $A$.

\section{Preliminaries} \label{s:pre}

This section is added for the sake of completeness.

\subsection{Finite Potent Endomorphisms}\label{ss:CT}

 Let $k$ be an arbitrary field, and let $V$ be a $k$-vector space.

    Let us now consider an endomorphism $\varphi$ of $V$. We say
that $\varphi$ is ``finite potent'' if $\varphi^n V$ is finite
dimensional for some $n$. This definition was introduced by J. Tate in \cite{Ta} as a basic tool for his elegant definition
of Abstract Residues.

 In 2007, M. Argerami, F. Szechtman and R. Tifenbach showed in \cite{AST} that an endomorphism $\varphi$ is
finite potent if and only if $V$ admits a $\varphi$-invariant
decomposition $V = U_\varphi \oplus W_\varphi$ such that
$\varphi_{\vert_{U_\varphi}}$ is nilpotent, $W_\varphi$ is finite
dimensional and $\varphi_{\vert_{W_\varphi}} \colon W_\varphi
\overset \sim \longrightarrow W_\varphi$ is an isomorphism.

Indeed, if $k[x]$ is the algebra of polynomials in the variable x with coefficients in $k$, we may view $V$ as an $k[x]$-module via $\varphi$, and the explicit definition of the above $\varphi$-invariant subspaces of $V$ is:
\begin{itemize}

\item $U_\varphi = \{v \in V \text{ such that } x^m v = 0 \text{ for some m }\}$;

\item $W_\varphi = \{v \in V \text{ such that } p(x) v = 0 \text{ for some } p(x) \in k[x] \text{ relative prime to } x\}$.

\end{itemize}

Note that if the annihilator polynomial of $\varphi$ is $x^m\cdot p(x)$ with $(x,p(x)) = 1$, then $U_\varphi = \Ker \varphi^m$ and $W_\varphi = \Ker p(\varphi)$.

Hence, this decomposition is unique. In this paper we shall call this decomposition the
$\varphi$-invariant AST-decomposition of $V$.

For a finite potent endomorphism $\varphi$, a trace $\tr_V(\varphi) \in k$ may
be defined as $\tr_V(\varphi) = \tr_{W_\varphi}(\varphi_{\vert_{W_\varphi}})$

This trace has the following properties:
\begin{enumerate}
\item if $V$ is finite dimensional, then $\tr_V(\varphi)$ is the
ordinary trace;
 \item if $W$ is a subspace of $V$ such that
$\varphi W \subset W$ then $$\tr_V(\varphi) = \tr_W(\varphi) +
\tr_{V/W}(\varphi)\, ;$$ \item if $\varphi$ is nilpotent, then
$\tr_V(\varphi) = 0$.
\end{enumerate}

Usually, $\tr_V$ is named ``Tate's trace''.

It is known that in general $\tr_V$ is not linear; that is, it is possible to find finite potent endomorphisms $\theta_1, \theta_2 \in \ed_k (V)$ such that $$\tr_V (\theta_1 + \theta_2) \ne \tr_V (\theta_1) + \tr_V (\theta_2)\, .$$

Moreover, with the previous notation, and using the AST-decomposition of $V$, D. Hern\'andez Serrano and the author of this paper have offered in \cite{HP} a definition
of a determinant for finite potent endomorphisms as follows:
$$\Det^k_V(1 +\varphi) := \Det^k_{W_\varphi}(1 + \varphi_{\vert_{W_\varphi}})\,.$$

This determinant satisfies the following properties:
\begin{itemize}
\item if $V$ is finite dimensional, then $\Det^k_V(1 + \varphi)$
is the ordinary determinant;

\item if $W$ is a subspace of $V$ such that $\varphi W \subset W$,
then $$\Det^k_V(1 + \varphi) = \Det^k_W(1 + \varphi) \cdot
\Det^k_{V/W}(1 + \varphi)\, ;$$

\item if $\varphi$ is nilpotent, then $\Det^k_V(1 + \varphi) = 1$.
\end{itemize}

For details readers are referred to  \cite{HP}, \cite{Pa1}, \cite{RPa} and \cite{Ta}.

\subsection{Jordan bases of nilpotent endomorphisms of infinite-dimensional vector spaces}\label{ss:nilpotent-basis}

Let $V$ be a vector space over an arbitrary field $k$ and let $f\in\ed_k (V)$ be a nilpotent endomorphism. 

If n is the nilpotency index of $f$, according to the statements \cite{Pa},  setting $W_i^f = \Ker f^i/[\Ker f^{i-1} + f(\Ker f^{i+1})]$ with $i\in \{1,2,\dots, n\}$, $\alpha_i (V,f) = \text{dim}_k W_i^f$ and $S_{\alpha_i (V,f)}$ a set such that $\# S_{\alpha_i (V,f)} = \alpha_i (V,f)$ with $S_{\alpha_i (V,f)} \cap S_{\alpha_j (V,f)} = \emptyset$ for all $i \ne j$, one has that there exists a family of vectors $\{{ {v_{s_i}}}\}$  that determines a Jordan basis of $f$: \begin{equation} \label{eq:jordan-basis-B} B = \underset {\begin{aligned} s_i &\in S_{{\alpha}_i (V,f)} \\ 1 &\leq i \leq n \end{aligned}} {\bigcup} \{{ {v_{s_i}}}, f ({{ v_{s_i}}}), \dots , f^{i-1} ({ {v_{s_i}}})\}\, .\end{equation}
Moreover, if we write $H_{s_i}^f = \langle { {v_{s_i}}}, f ({{ v_{s_i}}}), \dots , f^{i-1} ({ {v_{s_i}}}) \rangle$, the basis $B$ induces a decomposition \begin{equation} \label{eq:decomp} V =  \underset {\begin{aligned} s_i &\in S_{{\alpha}_i (V,f)} \\ 1 &\leq i \leq n \end{aligned}} \bigoplus H_{s_i}^f\, .\end{equation}

For a different method to construct Jordan bases of nilpotent endomorphisms of infinite-dimensional vector spaces readers can see \cite{LB}.

\subsection{Drazin inverse of Finite Potent Endomorphisms}\label{ss:drazin-inverse}

\subsubsection{Drazin Inverse of $(n\times n)$-matrices} \label{sss:Drazin-inv-finite}

\smallskip

Let $A\in {\text{Mat}}_{n\times n} ({\mathbb C})$.

\begin{defn} \label{def:index-fin} The ``index of $A$'', $i(A) \geq 0$, is the smallest integer such that $\rk (A^{i(A)}) = \rk (A^{i(A)+1})$.
\end{defn}

\smallskip

In 1958, given a matrix $A\in {\text{Mat}}_{n\times n} ({\mathbb C})$ with $i(A) = k$, M. P. Drazin -\cite{Dr}- showed the existence of a unique $(n\times n)$-matrix $A^D$ satisfying the equations:
\begin{itemize}
\item $A^{k+1} A^D = A^k$ for $k = i(A)$;
\item $A^D A A^D = A^D$;
\item $A^D A = A A^D$.
\end{itemize}

The Drazin inverse $A^D$ also verifies that

\begin{itemize}
\item $(A^D)^D = A$ if and only if $i(A) \leq 1$;
\item if $A^2 = A$, then $A^D = A$.
\end{itemize}

\smallskip

\subsubsection{Drazin inverse of Finite Potent Endomorphisms}\label{ss:drazin-inverse}

Let $V$  be an arbitrary $k$-vector space and let $\varphi \in \ed_k (V)$ be a finite potent endomorphism of $V$. Let us consider the AST-decomposition $V = U_\varphi \oplus W_\varphi$ induced by $\varphi$ (Subsection \ref{ss:CT}).

\smallskip

\begin{defn} \label{def:index-inf} We shall call ``index of $\varphi$'', $i(\varphi)$, to the nilpotent order of $\varphi_{\vert_{U_\varphi}}$.
\end{defn}

\smallskip

In \cite{Pa-Dr} (Lemma 3.2) is proved that for finite-dimensional vector spaces this definition of index coincides with Definition \ref{def:index-fin}. Note that $i(\varphi) = 0$ if and only if $V$ is a finite-dimensional vector space and $\varphi$ is an automorphism.

\smallskip

 For each finite potent endomorphism $\varphi$ there exists a unique finite potent endomorphism $\varphi^D$ that satisfies that:

\begin{enumerate}
\item $\varphi^{k+1} \circ \varphi^D = \varphi^{k}$;

\item $\varphi^D \circ \varphi \circ \varphi^D = \varphi^D$;

\item $\varphi^D \circ \varphi = \varphi \circ \varphi^D$,
\end{enumerate}

where $k$ is the index of $\varphi$.

The map $\varphi^D$ is the Drazin inverse of $\varphi$ and is the unique linear map such that:

 $$\varphi^D (v) = \left \{\begin{aligned} (\varphi_{\vert_{W_\varphi}})^{-1} \qquad &\text{ if } \qquad v \in W_\varphi \\ \qquad 0 \qquad \qquad &\text{ if } \qquad v \in U_\varphi \end{aligned} \right . \, .$$ 

Moreover, $\varphi^D$ satisfies the following properties:

\begin{itemize}

\item $(\varphi^D)^D = \varphi$ if and only if the $i(\varphi) \leq1$;

\item $\varphi = \varphi^D$ if and only if  $\varphi_{\vert_{U_\varphi}} = 0$ and $(\varphi_{\vert_{W_\varphi}})^2 = \text{Id}_{\vert_{W_\varphi}}$;

\item $\tr_V (\varphi + \varphi^D) = \tr_V (\varphi) + \tr_V(\varphi^D)$;

\item  if $\psi$ is a projection finite potent endomorphism, then $\psi^D  = \psi$.

\end{itemize}

\medskip

\subsection{CN Decomposition of a Finite Potent Endomorphism}

Given a finite potent endomorphism $\varphi \in \ed_k (V)$,  there exists a unique decomposition $\varphi = \varphi_1 + \varphi_2$, where $\varphi_1, \varphi_2 \in \ed_k (V)$ are finite potent endomorphisms satisfying that:

\begin{itemize}

\item $i(\varphi_1) \leq 1$;

\item $\varphi_2$ is nilpotent;

\item $\varphi_1 \circ \varphi_2 = \varphi_2 \circ \varphi_1 = 0$.

\end{itemize}

According to \cite{Pa-CN} -Theorem 3.2-, one has that $\varphi_1 = \varphi \circ \varphi^D \circ \varphi$, which is the core part of $\varphi$. Also, $\varphi_2$ is named the nilpotent part of $\varphi$.

Moreover, one has that \begin{equation} \label{eq:index1} \varphi = \varphi_1 \Longleftrightarrow U_\varphi = \Ker \varphi \Longleftrightarrow  W_\varphi = \text{ Im } \varphi \Longleftrightarrow (\varphi^D)^D = \varphi \Longleftrightarrow i(\varphi) \leq 1\, .\end{equation}

\medskip

\subsection{G-Drazin inverses of a square matrix}

Given $A \in {\text{Mat}}_{n\times n} ({\mathbb C})$ with $i(A) = r$, H. Wang and X. Liu introduced in \cite{WL} the notion of ``G-Drazin inverse'' of $A$ as a solution $X$ of the system \begin{equation} \label{eq:G-D-2} \begin{aligned} AXA &= A\, ; \\ XA^{r+1}&= A^r \, ;\\ A^{r+1} X &= A^r\, ,\end{aligned}\end{equation} \noindent where $X$ is a $(n\times n)$-matrix with entries in $\mathbb C$. 

Recently,  C. Coll, M.  Lattanzi and N.  Thome have proved in \cite{CLT} that a matrix $X  \in {\text{Mat}}_{n\times n} ({\mathbb C})$ is a solution of the system (\ref{eq:G-D-2}) if and only if $X$ satisfies that  \begin{equation} \label{eq:G-D-2} \begin{aligned} AXA &= A\, ; \\ XA^{r}&= A^rX\, .\end{aligned} \end{equation}

Usually, the set of G-Drazin inverses of a matrix $A$ is denoted by $A{GD}$ and a G-matrix inverse of $A$ is denoted by $A^{GD}$.

If $J$ is the Jordan matrix associated with $A \in {\text{Mat}}_{n\times n} ({\mathbb C})$, such that $A = B\cdot J \cdot B^{-1}$, with $B$ being a non-singular matrix and $$J = \begin{pmatrix} J_1 & 0 \\ 0 & J_0\end{pmatrix}\, ,$$\noindent $J_0$ and $J_1$ being the parts of $J$ corresponding to zero and non-zero eigenvalues respectively, it is known that a $G$-Drazin inverse is \begin{equation} \label{eq:cls} A^{GD} = B \cdot \begin{pmatrix} J_1^{-1} & 0 \\ 0 & J_0^- \end{pmatrix} \cdot B^{-1}\, ,\end{equation}

where is $J_0^-$ is a generalized inverse of $J_0$ (1-inverse).

\medskip

\section{G-Drazin inverses of finite potent endomorphisms} \label{s:GDrazin-Inv-FP}

The aim of this section is to generalize the definition and the main properties of the G-Drazin inverses of a matrix $A$ to finite potent endomorphisms.

Let $k$ be a field and let $V$ be an arbitrary $k$-vector space.

\begin{defn} \label{d:G-Drazin-FP} Given a finite potent endomorphism $\varphi \in \ed_k (V)$, we say that an endomorphism $\varphi^{GD} \in \ed_k (V)$ is a G-Drazin inverse of $\varphi$ when it satisfies that 

\begin{equation} \label{eq:G-D-fp} \begin{aligned} \varphi \circ \varphi^{GD} \circ \varphi &= \varphi \, ; \\ \varphi^{GD} \circ \varphi^{r}&= \varphi^r \circ \varphi^{GD}\, ,\end{aligned} \end{equation} \noindent where $i(\varphi) = r$.

\end{defn}

\smallskip

\begin{lem} \label{l:G-D-inv-345} Let $\varphi \in \ed_k (V)$ be a finite potent endomorphism with $i(\varphi) = r$ and let $V = W_\varphi \oplus U_\varphi$ be the AST-decomposition determined by $\varphi$. If $f\in \ed_k(V)$ is an endomorphism such that $f \circ \varphi^{r}= \varphi^r \circ f$, then $W_\varphi$ and $U_\varphi$ are invariant under the action of $f$.
\end{lem}

\begin{proof} Let $\varphi = \varphi_1 + \varphi_2$ be the CN-decomposition of $\varphi$. 

Since $i(\varphi) = r$, bearing in mind that $\varphi^r = \varphi_1^r$, $(\varphi_1^r)_{\vert_{W_\varphi}} \in \aut_k (W_\varphi)$ and \linebreak Im $\varphi_1^r = W_{\varphi}$, if $w\in W_{\varphi}$ and $\varphi_1^r (w') = w$,
 then $$f (w) = (f \circ \varphi_1^r) (w') = (\varphi_1^r \circ f) (w') \in W_\varphi\, .$$

Accordingly, $W_\varphi$ is $f$-invariant.

Moreover, if $u\in U_\varphi$, then $$0 = (f \circ \varphi_1^r) (u) = (\varphi_1^r \circ f) (u)$$\noindent and we deduce that $f (u) \in \Ker  \, \varphi_1^r = U_\varphi$. Hence, $U_\varphi$ is also $f$-invariant.
\end{proof}

\begin{cor} \label{c:G-D-inv-3453} If $\varphi^{GD} \in \ed_k (V)$ is a G-Drazin inverse of a finite potent endomorphism $\varphi \in \ed_k (V)$, with $i(\varphi) = r$ and AST-decomposition $V = W_\varphi \oplus U_\varphi$, then $W_\varphi$ and $U_\varphi$ are invariant under the action of  $\varphi^{GD}$.
\end{cor}

\smallskip

The structure of a G-Drazin inverse of a finite potent endomorphism is given by the following proposition:

\begin{prop} \label{p:char-G-D-fp} Given a finite potent endomorphism $\varphi \in \ed_k (V)$  with $i(\varphi) = r$ and AST-decomposition $V = W_\varphi \oplus U_\varphi$, one has that $\varphi^{GD} \in \ed_k (V)$ is a G-Drazin inverse of $\varphi$ if and only if  $W_\varphi$ and $U_\varphi$ are invariant under the action of  $\varphi^{GD}$, $(\varphi^{GD})_{\vert_{W_\varphi}} = (\varphi_{\vert_{W_\varphi}})^{-1}$ and $(\varphi^{GD})_{\vert_{U_\varphi}} = (\varphi_{\vert_{U_\varphi}})^{-}$, where $ (\varphi_{\vert_{U_\varphi}})^{-}$ is a generalized inverse of $\varphi_{\vert_{U_\varphi}}$.
\end{prop}

\begin{proof} Let $\varphi^{GD} \in \ed_k (V)$ be a G-Drazin inverse of $\varphi$. It follows from Corollary \ref{c:G-D-inv-3453} that  $W_\varphi$ and $U_\varphi$ are invariant under the action of  $\varphi^{GD}$. Also, since $\varphi \circ \varphi^{GD} \circ \varphi = \varphi$, it is clear that $(\varphi^{GD})_{\vert_{W_\varphi}} = (\varphi_{\vert_{W_\varphi}})^{-1}$ and $(\varphi^{GD})_{\vert_{U_\varphi}} = (\varphi_{\vert_{U_\varphi}})^{-}$, where $ (\varphi_{\vert_{U_\varphi}})^{-}$ is a generalized inverse of $\varphi_{\vert_{U_\varphi}}$.

Conversely, if  we consider an endomorphism $\psi \in \ed_k (V)$ satisfying that $W_\varphi$ and $U_\varphi$ are invariant under the action of $\psi$,  $\psi_{\vert_{W_\varphi}} = (\varphi_{\vert_{W_\varphi}})^{-1}$ and $\psi_{\vert_{U_\varphi}} = (\varphi_{\vert_{U_\varphi}})^{-}$, with $ (\varphi_{\vert_{U_\varphi}})^{-}$ a generalized inverse of $\varphi_{\vert_{U_\varphi}}$, it is easy to check that $\varphi \circ \psi \circ \varphi = \varphi$.

Moreover, since $(\varphi^{r})_{\vert_{U_\varphi}} = 0$, then it follows from the properties of $\psi$ that $$(\psi \circ \varphi^{r})_{\vert_{W_\varphi}} =  (\varphi^{r-1})_{\vert_{W_\varphi}} = (\varphi^r \circ \psi)_{\vert_{W_\varphi}}$$\noindent from where we deduce that $$\psi \circ \varphi^{r} =  \varphi^r \circ \psi$$\noindent and the statement is proved.
\end{proof} 

Direct consequences of Proposition \ref{p:char-G-D-fp} are:

\begin{cor} \label{c:char-G-D-fp} Given a finite potent endomorphism $\varphi \in \ed_k (V)$ with CN-decomposition $\varphi = \varphi_1 + \varphi_2$, if $\varphi^{GD} \in \ed_k (V)$ is a G-Drazin inverse of $\varphi$, one has that $\varphi^{GD} = \varphi^D + \varphi_2^{GD}$, where $\varphi^D$ is the Drazin inverse of $\varphi$, and $\varphi_2^{GD} \in \ed_k (V)$ is the unique linear map satisfying that $$\varphi_2^{GD} (v) = \left \{ \begin {aligned}  0\quad  \quad &\text{ if } v \in W_\varphi \\ (\varphi^{GD})_{\vert_{U_\varphi}} &\text{ if } v \in U_\varphi\end{aligned} \right .\, ,$$\noindent which is a G-Drazin inverse of $\varphi_2$.
\end{cor}

\begin{cor} \label{c:Dra-G-d-1-38} If $\varphi = \varphi_1 + \varphi_2$ is the CN-decomposition of a finite potent endomorphism $\varphi \in \ed_k (V)$, then the Drazin inverse $\varphi^D$ is a G-Drazin inverse of $\varphi_1$.
\end{cor}

\medskip

We shall now characterize all of the G-Drazin inverses of a finite potent endomorphism.

\begin{lem} \label{l:char-gen-inv-hom-n2} Let $E$ be a $k$-vector space of dimension n and let $f\in \ed_k (E)$ be an endomorphism with annihilating polynomial $a_f (x) )= x^n$. If $e\in E$ is a vector such that $f^{n-1} (e) \ne 0$, then every generalized inverse $f^- \in \ed_k (E)$ is determined by the expressions $$f^- (f^i (e)) = \left \{ \begin{aligned} f^{i-1} (e) + \lambda_i f^{n-1} (e) \quad &\text{ if } \quad i\geq 1 \\ \quad {\tilde e} \quad \quad \quad \quad  \quad &\text{ if } \quad i = 0 \end{aligned} \right . \, ,$$\noindent with $i\in \{0,1, \dots, n-1\}$, $\lambda_i \in k$ for every $i\in \{1, \dots, n-1\}$ and $\tilde e\in E$ being an arbitrary vector.
\end{lem}

\begin{proof} Since $f \circ f^- \circ f = f$, we have that $f^-$ is a generalized inverse of $f$ if and only if $(f \circ f^-)_{\vert_{\text{Im f}}} = \text{Id}_{\vert_{\text{Im f}}}$. Hence, since $e\notin \text{Im } f$ we have that $f^- (e) = {\tilde e}$, where ${\tilde e} \in E$ is an arbitrary vector.

Moreover, bearing in mind that $(f\circ f^-) (f^i (e)) = f^i (e)$ for all $i\geq 1$, one has that $$f^- (f^i (e)) \in f^{-1} (f^i (e)) + \Ker f\, ,$$\noindent and we get that $$f^- (f^i(e)) = f^{i-1} (e) + \lambda_i f^{n-1} (e)$$\noindent for all $i\in \{1, \dots, n-1\}$.

Accordingly, since $\{e, f(e), \dots, f^{n-1} (e)$ is a Jordan basis of $E$ induced by $f$, the claim is deduced.
\end{proof}

We can reformulate the statement of Lemma \ref{l:char-gen-inv-hom-n2} as follows:

\begin{cor} \label{c:char-gen-inv-hom-n} Let $E$ be a $k$-vector space of dimension n and let $f\in \ed_k (E)$ be an endomorphism with annihilating polynomial $a_f (x) )= x^n$. If $e\in E$ is a vector such that $f^{n-1} (e) \ne 0$, then every generalized inverse $f^- \in \ed_k (E)$ is determined by the expressions \begin{equation} \label{eq:expl-1-inv-nilpo-lm} f^- (f^i (e)) = \left \{ \begin{aligned} f^{i-1} (e) + \lambda_i f^{n-1} (e) \quad &\text{ if } \quad i\geq 1 \\ \sum_{h= 0}^{n-1} \alpha_h f^h (e) \quad &\text{ if } \quad i = 0 \end{aligned} \right . \, ,\end{equation}\noindent with $i\in \{0, \dots, n-1\}$, $\lambda_i, \alpha_h \in k$ for every $i\in \{1, \dots, n-1\}$ and $h \in \{0, \dots, n-1\}$.
\end{cor}

Furthermore, similar to Lemma \ref{l:char-gen-inv-hom-n2} one can prove that

\begin{lem} \label{l:char-gen-inv-hom-n2} If $E$ is a finite-dimensional $k$-vector space, $f\in \ed_k (E)$ is an endomorphism with annihilating polynomial $a_f (x) )= x^n$ and $$\bigcup_{j= 1}^r \{e_j, f (e_j), \dots, f^{n_j - 1} (e_j)\}$$\noindent is a Jordan basis of $E$ induced by $f$, then  every generalized inverse $f^- \in \ed_k (E)$ is determined by the expressions \begin{equation} \label{eq:expl-1-inv-nilpo-lm-2} f^- (f^i (e_j)) = \left \{ \begin{aligned} f^{i-1} (e_j) + \sum_{s=1}^r \lambda_{j,i}^s f^{n_s-1} (e_s) \quad &\text{ if } \quad i\geq 1 \\  \sum_{s=1}^r [\sum_{h= 0}^{n_s-1} \alpha_{j,h}^s f^h (e_s)] \quad &\text{ if } \quad i = 0 \end{aligned} \right . \, ,\end{equation}\noindent with $\lambda_{j,i}^s,  \alpha_{j,h}^s \in k$ for each $j,s\in \{1, \dots, r\}$, $i\in \{0, \dots, n_j-1\}$ and $h \in \{0, \dots, n_s-1\}$.

\end{lem}

Let us again consider a finite potent endomorphism $\varphi \in \ed_k (V)$ of an arbitrary $k$-vector space $V$ with $i(\varphi) = r$. If $V = W_\varphi \oplus U_\varphi$ is the AST-decomposition induced by $\varphi$, let

 \begin{equation} \label{eq:jordan-basis-B-fp} B_\varphi = \underset {\begin{aligned} s_h &\in S_{{\alpha}_h (V,\varphi_{\vert_{U_\varphi}})} \\ 1 &\leq h \leq r \end{aligned}} {\bigcup} \{{ {v_{s_h}}}, \varphi ({{ v_{s_h}}}), \dots , \varphi^{h-1} ({ {v_{s_h}}})\}\end{equation}

be a Jordan basis of $U_\varphi$ induced by $\varphi_{\vert_{U_\varphi}}$ (see Subsection \ref{ss:nilpotent-basis}).

Let us now denote $${\overline S}_{\alpha, \varphi} = S_{\alpha_1 (V,\varphi_{\vert_{U_\varphi}})} \cup \dots \cup S_{\alpha_r (V,\varphi_{\vert_{U_\varphi}})}$$\noindent and $\beta (V,\varphi) = \# {\overline S}_{\alpha, \varphi}$.

Accordingly, similar to Lemma \ref{l:char-gen-inv-hom-n2}, it is easy to check that

\begin{lem} \label{l:char-gen-f-p-basi-r-345} Given  an arbitrary $k$-vector space $V$ and a finite potent endomorphism $\varphi \in \ed_k (V)$ with $i(\varphi) = r$ and AST-decomposition $V = W_\varphi \oplus U_\varphi$, fixing a Jordan basis $B_\varphi$ of  $U_\varphi$ as in (\ref{eq:jordan-basis-B-fp}), then  every generalized inverse $(\varphi_{\vert_{U_\varphi}})^- \in \ed_k (U_\varphi)$ is determined by the expressions $$(\varphi_{\vert_{U_\varphi}})^- (\varphi^i (v_{s_h})) = \left \{ \begin{aligned} \varphi^{i-1} (v_{s_h}) + \sum_{s_t \in {\overline S}_{\alpha, \varphi}} \lambda^i_{s_h,s_t} \varphi^{t-1} (v_{s_t}) \quad &\text{ if } \quad 1\leq  i\leq h-1 \\ \quad {v} \quad \quad \quad \quad  \quad  \quad \quad \quad \quad &\text{ if } \quad i = 0 \end{aligned} \right . \, ,$$\noindent with $v\in U_\varphi$, $1 \leq h \leq r$, $i\in \{0,1, \dots, h-1\}$, $\lambda^i_{s_h, s_t} \in k$ and $\lambda^i_{s_h, s_t} = 0$ for almost all $s_t \in {\overline S}_{\alpha, \varphi}$ (for each $s_h \in S_{{\alpha}_h (V,\varphi_{\vert_{U_\varphi}})}$).
\end{lem}

Writing $$(\varphi_{\vert_{U_\varphi}})^- (v_{s_h}) =   \underset {\begin{aligned} s_l &\in {\overline S}_{\alpha, \varphi} \\  0 \leq j &\leq l-1 \end{aligned}} \sum \alpha_{s_h,s_l}^j \varphi^j (v_{s_l}) \in U_\varphi\, ,$$\noindent one has that

\begin{cor} \label{c:char-gen-f-p-basi-r-345} Given  an arbitrary $k$-vector space $V$ and a finite potent endomorphism $\varphi \in \ed_k (V)$ with $i(\varphi) = r$ and AST-decomposition $V = W_\varphi \oplus U_\varphi$, fixing a Jordan basis $B_\varphi$ of  $U_\varphi$ as in (\ref{eq:jordan-basis-B-fp}), then  every generalized inverse $(\varphi_{\vert_{U_\varphi}})^- \in \ed_k (U_\varphi)$ is determined by the expressions $$(\varphi_{\vert_{U_\varphi}})^- (\varphi^i (v_{s_h})) = \left \{ \begin{aligned} \varphi^{i-1} (v_{s_h}) + \sum_{s_t \in {\overline S}_{\alpha, \varphi}} \lambda^i_{s_h,s_t} \varphi^{t-1} (v_{s_t}) \quad &\text{ if } \quad 1\leq  i\leq h-1 \\ \quad \underset {\begin{aligned} s_l &\in {\overline S}_{\alpha, \varphi} \\ 0 \leq j &\leq l-1 \end{aligned}} \sum \alpha_{s_h,s_l}^j \varphi^j (v_{s_l}) \quad \quad &\text{ if } \quad i = 0 \end{aligned} \right . \, ,$$\noindent with $1 \leq h \leq r$, $i\in \{0,1, \dots, h-1\}$, $\lambda^i_{s_h, s_t}, \alpha_{s_h, s_l}^j \in k$, and $\lambda^i_{s_h,s_t} = 0 = \alpha_{s_h, s_l}^j$ for almost all $s_t, s_l \in {\overline S}_{\alpha, \varphi}$ and $0 \leq j \leq l-1$.
\end{cor}

\begin{cor} \label{c:char-0-GD-inv-2345} With the notation of Corollary \ref{c:char-G-D-fp}, $\varphi_2^{GD} = 0$ if and only if $i (\varphi) \leq 1$.
\end{cor}

\begin{proof} It follows from Lemma \ref{l:char-gen-f-p-basi-r-345} that $\varphi_2^{GD} = 0$ if and only if $\Ker \varphi = U_\varphi$, from where the claim is proved.
\end{proof}

\begin{cor} Given a finite potent endomorphism $\varphi \in \ed_k (V)$, one has that the Drazin inverse $\varphi^D \in \ed_k (V)$ is a G-Drazin inverse of $\varphi$ if and only if $i (\varphi) \leq 1$.
\end{cor}

\begin{proof} If $\varphi = \varphi_1 + \varphi_2$ is the CN-decomposition, then according to Corollary \ref{c:char-G-D-fp} one has that $\varphi^D \in \ed_k (V)$ is a G-Drazin inverse of $\varphi$ if and only if $0$ is a G-Drazin inverse of $\varphi_2$ and, bearing in mind Corollary \ref{c:char-0-GD-inv-2345}, the statement is deduced.
\end{proof}

Accordingly, from Proposition \ref{p:char-G-D-fp} and Lemma \ref{l:char-gen-f-p-basi-r-345} we have characterized all the G-Drazin inverses of a finite potent endomorphism $\varphi \in \ed_k (V)$. Note that, in general, a G-Drazin inverse $\varphi^{GD} \in \ed_k (V)$ is not a finite potent endomorphism.

If we denote by $X_\varphi^{\{GD\}}$ the set of all G-Drazin inverses of a finite potent endomorphism $\varphi \in \ed_k (V)$, fixing a Jordan basis $B_\varphi$ of  $U_\varphi$ as in (\ref{eq:jordan-basis-B-fp}), with the notation of Subsection \ref{ss:nilpotent-basis}, $i(\varphi) = r$ and \noindent we have a bijection \begin{equation} \label{eq:bij-chr-g-dr-3487} \begin{aligned} X_\varphi^{\{GD\}} &\overset {\sim} \longrightarrow \prod_{h= 1}^r \big (\prod_{\alpha_h (V,\varphi_{\vert_{U_\varphi}})}\big [ U_\varphi \times  \prod_{i= 1}^{h-1} (\bigoplus_{\beta (V,\varphi)} k) \big ] \big )  \\ \varphi^{GD} &\longmapsto \big ( \big (\varphi^{GD} (s_h), ((\lambda^i_{s_h,s_t})_{s_t \in {\overline S}_{\alpha, \varphi}})_{1\leq i\leq h-1}\big )_{s_h \in S_{\alpha_h (V,\varphi_{\vert_{U_\varphi}})}} \big )_{h\in \{1,\dots, r\}} \end{aligned}\, .\end{equation}

\begin{exam} \label{ex:G-Dr-inf-ir} Let $k$ be an arbitrary ground field, let $V$ be a $k$-vector space of countable dimension over $k$ and let $\{{ {v_1}},{ {v_2}},{ {v_3}},\dots\}$ be a basis of $V$ indexed by the natural numbers. 

Let $\varphi \in \ed_{k} (V)$ the finite potent endomorphism defined as follows: 
$$\varphi ({ {v_i}}) = \left \{\begin{aligned} { {v_{2}}} +  { {v_{5}}}   +  { {v_{7}}} \quad   &\text{ if } \quad i = 1 \\   { {v_{1}}} +  3{ {v_{2}}} \quad   \quad   &\text{ if } \quad i = 2 \\    { {v_{4}}}  \qquad   \quad   &\text{ if } \quad i = 3 \\   { {v_{1}}} - { {v_{3}}} \quad   \quad   &\text{ if } \quad i = 4 \\ -{ {v_{3}}} +  2{ {v_{5}}}   +  2{ {v_{7}}} \quad   &\text{ if } \quad i = 5 \\    3{ {v_{i+1}}} \qquad \quad  &\text{ if } \quad i = 5h + 1 \\  { {0}} \qquad \quad   &\text{ if } \quad i = 5h + 2 \\ -{ {v_{i-2}}} + 2{ {v_{i+1}}}  \quad  &\text{ if } \quad i = 5h + 3 \\ { {v_{i-2}}} + { {v_{i+1}}} \quad   &\text{ if } \quad i = 5h + 4 \\ - { {v_{i-4}}} + 5{ {v_{i-3}}}  \quad  &\text{ if } \quad i = 5h + 5    \end{aligned} \right .$$\noindent for all $h\geq 1$.

We have that the AST-decomposition $V = U_\varphi \oplus W_\varphi$ is determined by the subspaces $$W_\varphi = \langle  { {v_{1}}}, \,  { {v_{2}}}, \,  { {v_{3}}}, \,  { {v_{4}}}, \,  { {v_{5}}} +  { {v_{7}}}\rangle \text{ and }  U_\varphi = \langle { {v_{j}}} \rangle_{j\geq 6}\, .$$

In this basis of $W_\varphi$ one has that $$\varphi_{\vert_{W_\varphi}}\equiv A_{{W_\varphi}}=\begin{pmatrix} 0 & 1 & 0 & 1 & 0 \\ 1 & 3 & 0 & 0 & 0 \\ 0 & 0 & 0 & -1 & -1 \\ 0 & 0 & 1 & 0 & 0 \\ 1 & 0 & 0 & 0 & 2 \end{pmatrix}\, ,$$\noindent and, computing $A_{{W_\varphi}}^{-1}$, we obtain that the Drazin inverse of $\varphi$ is $$\varphi^D ({ {v_i}}) = \left \{\begin{aligned} 6v_1 -  2v_{2} + 3v_4 - 3v_{5}\quad   &\text{ if } \quad i = 1 \\   -2v_{1} +  v_{2} -v_4 + v_5 \quad   \quad   &\text{ if } \quad i = 2 \\  6v_1 - 2v_2 + 2 v_{4} -3v_5 \quad   &\text{ if } \quad i = 3 \\   v_{3} \quad   \quad \qquad \quad   &\text{ if } \quad i = 4 \\  3v_1 - v_2 + v_{4} - v_{5} \quad   &\text{ if } \quad i = 5 \\  0 \quad   \quad \qquad \quad   &\text{ if } \quad i \geq 6 \end{aligned} \right . \, . $$

Moreover,  we can write $U_\varphi = \underset {i\geq 2} \bigoplus H_i$ with $H_i=\langle {v_{5i - 4}},{v_{5i - 3}},{v_{5i - 2}}, v_{5i - 1}, v_{5i}\rangle$ for all $i\geq 2$, and in the same bases we have that $$\varphi_{\vert_{H_i}}\equiv A_{H_i} = \begin{pmatrix}
0 & 0 & -1 & 0 & -1 \\
3 & 0 & 0 & 1 & 5 \\
0 & 0 & 0 & 0 & 0 \\
0 & 0 & 2 & 0 & 0 \\
0 & 0 & 0 & 1 & 0 
\end{pmatrix}$$\noindent  for all $i\geq 2$.

For every $i \geq 2$, one has that $$\{v_{5i - 2}, - v_{5i - 4} + 2 v_{5i - 1}, -v_{5i - 3} + 2v_{5i}, -2v_{5i - 4} + 10 v_{5i - 3}, -6 v_{5i - 3}\}$$\noindent is a Jordan basis of $H_i$ induced by $\varphi_{\vert_{H_i}}$ and, therefore, $$\varphi_{\vert_{H_i}} \equiv  P \cdot \begin{pmatrix}
0 & 0 & 0 & 0 & 0 \\
1 & 0 & 0 & 0 & 0 \\
0 & 1 & 0 & 0 & 0 \\
0 & 0 & 1 & 0 & 0 \\
0 & 0 & 0 & 1 & 0 
\end{pmatrix} \cdot P^{-1}$$ with $$P = \begin{pmatrix} 0 &  -1 & 0 & -2 & 0 \\ 0 & 0 & -1 & 10 & -6 \\ 1 & 0 & 0 & 0 & 0 \\ 0 & 2 & 0 & 0  & 0 \\ 0 & 0 & 2 & 0 & 0 \end{pmatrix}\, .$$

Accordingly, it follows from Corollary \ref{c:char-G-D-fp} and Corollary \ref{c:char-gen-f-p-basi-r-345} that $\varphi_2^{GD}\in \ed_k (V)$ is the unique linear map such that:

\begin{itemize}

\item $\varphi_2^{GD} (v_i) = 0$ for $i\in \{1,2,3,4\}$;

\item $\varphi_2^{GD} (v_5 + v_7) = 0$;

\item $\varphi_2^{GD}(v_{5i - 2}) = \sum_{j\geq 6} \alpha_{i,j} v_j$

\item $\varphi_2^{GD}(- v_{5i - 4} + 2 v_{5i - 1}) = v_{5i - 2} + \sum_{h\geq 2} \lambda^1_{i, 5h-3} v_{5h-3}$;

\item $\varphi_2^{GD}(-v_{5i - 3} + 2v_{5i}) = - v_{5i - 4} + 2 v_{5i - 1} + \sum_{h\geq 2} \lambda^2_{i, 5h-3} v_{5h-3}$;

\item $\varphi_2^{GD}( -2v_{5i - 4} + 10 v_{5i - 3}) = -v_{5i - 3} + 2v_{5i} + \sum_{h\geq 2} \lambda^3_{i, 5h-3} v_{5h-3}$;

\item $\varphi_2^{GD}( -6 v_{5i - 3}) =  -2v_{5i - 4} + 10 v_{5i - 3} + \sum_{h\geq 2} \lambda^4_{i, 5h-3} v_{5h-3}$;

\end{itemize}

for every $i\geq 2$, and with $\alpha_{i,j} = 0$ for almost all $j\geq 6$ (for each $i \geq 2$) and $\lambda^i_{5h-3,s} = 0$ for almost all $h\geq 2$ (for every $i\geq 2$ and for every $s\in \{2,3,4,5\}$).

Thus, a non-difficult computation shows that  $\varphi_2^{GD}\in \ed_k (V)$ is the unique linear map such that:

\begin{itemize}

\item $\varphi_2^{GD} (v_i) = 0$ for $i\in \{1,2,3,4\}$;

\item $\varphi_2^{GD} (v_5) =  -\frac13 v_{6} + \frac53 v_{7} + \frac16 \sum_{h\geq 2} \lambda^2_{5h-3,5} v_{5h-3}$;

\item $\varphi_2^{GD}(v_{5i - 4}) = -\frac53 v_{5i - 4} - \frac{47}6 v_{5i - 3} - v_{5i} - \sum_{h\geq 2} (\frac12 \lambda^3_{i, 5h-3} + \frac56 \lambda^4_{i, 5h-3}) v_{5h-3}$; 

\item $\varphi_2^{GD}(v_{5i - 3}) = \frac13 v_{5i - 4} - \frac53 v_{5i-3} - \frac16 \sum_{h\geq 2} \lambda^4_{i, 5h-3} v_{5h-3}$;

\item $\varphi_2^{GD}(v_{5i - 2}) =  \sum_{j\geq 6} \alpha_{i,j} v_j$;

\item $\varphi_2^{GD}(v_{5i - 1}) = \frac56 v_{5i-4} - \frac{47}{12} v_{5i - 3} +\frac12 v_{5i-2} - \frac12 v_{5i} +$ \newline $$\sum_{h\geq 2}( \frac12 \lambda^1_{i, 5h-3} - \frac14 \lambda^3_{i, 5h-3} v_{5h-3} - \frac5{12} \lambda^4_{i, 5h-3})\, ;$$

\item $\varphi_2^{GD}(v_{5i}) =  -\frac13 v_{5i - 4} - \frac5{6} v_{5i - 3} + v_{5i - 1} - \frac56 v_{5i} + \sum_{h\geq 2} (\frac12 \lambda^2_{i, 5h-3} - \frac1{12} \lambda^4_{i, 5h-3}) v_{5h-3}$;

\end{itemize}

for every $i\geq 2$, and with $\alpha_{i,j} = 0$ for almost all $j\geq 6$ (for each $i \geq 2$) and $\lambda^i_{5h-3,s} = 0$ for almost all $h\geq 2$ (for every $i\geq 2$ and for every $s\in \{2,3,4,5\}$).

Hence, bearing in mind that $\varphi^{GD} = \varphi^D + \varphi_2^{GD}$ (Corollary \ref{c:char-G-D-fp}), one has that a G-Drazin inverse $\varphi^{GD}$ is determined by:

\begin{itemize}

\item $\varphi^{GD} (v_1) = 6v_1 -  2v_{2} + 3v_4 - 3v_{5}$;

\item $\varphi^{GD} (v_2) = -2v_{1} +  v_{2} -v_4 + v_5$;

\item $\varphi^{GD} (v_3) = 6v_1 - 2v_2 + 2 v_{4} -3v_5$;

\item $\varphi^{GD} (v_4) = v_3$;

\item $\varphi^{GD} (v_5) = 3v_1 - v_2 + v_{4} - v_{5}  -\frac13 v_{6} + \frac53 v_{7} + \frac16 \sum_{h\geq 2} \lambda^2_{5h-3,5} v_{5h-3}$;

\item $\varphi^{GD}(v_{5i - 4}) = -\frac53 v_{5i - 4} - \frac{47}6 v_{5i - 3} - v_{5i} - \sum_{h\geq 2} (\frac12 \lambda^3_{i, 5h-3} + \frac56 \lambda^4_{i, 5h-3}) v_{5h-3}$; 

\item $\varphi^{GD}(v_{5i - 3}) = \frac13 v_{5i - 4} - \frac53 v_{5i-3} - \frac16 \sum_{h\geq 2} \lambda^4_{i, 5h-3} v_{5h-3}$;

\item $\varphi^{GD}(v_{5i - 2}) =  \sum_{j\geq 6} \alpha_{i,j} v_j$;

\item $\varphi^{GD}(v_{5i - 1}) = \frac56 v_{5i-4} - \frac{47}{12} v_{5i - 3} +\frac12 v_{5i-2} - \frac12 v_{5i} +$\newline $$\sum_{h\geq 2}( \frac12 \lambda^1_{i, 5h-3} - \frac14 \lambda^3_{i, 5h-3} v_{5h-3} - \frac5{12} \lambda^4_{i, 5h-3})\, ;$$

\item $\varphi^{GD}(v_{5i}) =  -\frac13 v_{5i - 4} - \frac5{6} v_{5i - 3} + v_{5i - 1} - \frac56 v_{5i} + \sum_{h\geq 2} (\frac12 \lambda^2_{i, 5h-3} - \frac1{12} \lambda^4_{i, 5h-3}) v_{5h-3}$;

\end{itemize}

for every $i\geq 2$, and with $\alpha_{i,j} = 0$ for almost all $j\geq 6$ (for each $i \geq 2$) and $\lambda^i_{5h-3,s} = 0$ for almost all $h\geq 2$ (for every $i\geq 2$ and for every $s\in \{2,3,4,5\}$).
\end{exam}

\begin{rem} Note that the explicit expression of the bijection (\ref{eq:bij-chr-g-dr-3487}) for the finite potent endomorphism $\varphi$ of Example \ref{ex:G-Dr-inf-ir} is $$\begin{aligned} X_\varphi^{\{GD\}} &\overset {\sim} \longrightarrow \prod_{i\in \mathbb N}\big ( U_\varphi \times  \prod_{j= 1}^{4} (\bigoplus_{h\in \mathbb N} k) \big )  \\ \varphi^{GD} &\longmapsto\big  (\, \varphi^{GD}(v_{5i - 2}), ((\lambda^j_{i, 5h-3})_{h\in \mathbb N})_{1\leq j \leq 4}\, \big )_{i\in \mathbb N} \end{aligned}\, .$$

\end{rem}

To finish this section, we shall briefly study the G-Drazin inverses of a finite potent endomorphism that also are finite potent.

Let $V$ be again an arbitrary $k$-vector space and let $\varphi \in \ed_k (V)$ be a finite potent endomorphism.

With the above notation, if we denote by $(\varphi^{GD})_{\{\lambda_{s_h}^i, \alpha_{s_h}^j\}}$ to the unique linear map of $V$ such that $$(\varphi^{GD})_{\{\lambda_{s_h}^i, \alpha_{s_h}^j\}} (v) = \left \{\begin{aligned}   (\varphi_{\vert_{W_\varphi}})^{-1} (v) \quad &\text{ if } v\in W_\varphi \\ (\varphi_{\vert_{U_\varphi}})^{-} (v) \quad &\text{ if } v\in U_\varphi \end{aligned} \right . \, ,$$\noindent where $(\varphi_{\vert_{U_\varphi}})^{-}$ is the generalized inverse of $\varphi_{\vert_{U_\varphi}} \in \ed_k (U_\varphi)$ characterized in Corollary \ref{c:char-gen-f-p-basi-r-345}, it is clear that $(\varphi^{GD})_{\{\lambda_{s_h}^i, \alpha_{s_h}^j\}}$ is finite potent when \begin{equation} \label{eq:con-GD-D-f-p} \lambda_{s_h}^i = 0 = \alpha_{s_h}^j \, \text{ for almost all } \, i, \, j \, \text{ and }\, s_h\, .\end{equation}

	It is clear that the condition (\ref{eq:con-GD-D-f-p}) is sufficient for determining that $(\varphi^{GD})_{\{\lambda_{s_h}^i, \alpha_{s_h}^j\}}$ is a finite potent endomorphism, but this condition is not necessary for this fact as it is immediately deduced from the following counter-example: given a countable $k$-vector space $V$ with a basis $\{{ {v_1}},{ {v_2}},{ {v_3}},\dots\}$ indexed by the natural numbers, if we consider the finite potent endomorphism $\varphi \in \ed_k (V)$ defined as $$\varphi (v_i) = \left \{ \begin{aligned} v_1 + v_2 \quad &\text{ i } i = 1 \\ v_1 + 2v_2\quad &\text{ if } i = 2 \\ v_{i+1} \quad \quad   &\text{ if } i = 2j+1 \\ 0\, \,  \quad\quad  &\text{ if } i = 2j + 2 \end{aligned} \right .$$\noindent for every $j\geq 1$, then it is clear that  $$\varphi^{GD} (v_i) = \left \{ \begin{aligned} 2 v_1 - v_2 \quad &\text{ i } i = 1 \\ -v_1 + v_2\quad &\text{ if } i = 2 \\- v_i -  v_{i+1}  \quad  &\text{ if } i = 2j+1 \\ v_{i-1} + v_i \quad  &\text{ if } i = 2j + 2 \end{aligned} \right .$$\noindent for every $j\geq 1$, is a G-Drazin inverse of $\varphi$ that does not satisfy the condition  (\ref{eq:con-GD-D-f-p}).

\begin{rem} A remaining problem is obtaining a computable method for determining when a G-Drazin inverse $\varphi^{GD}$ of a finite potent endomorphism $\varphi$ is also finite potent.
\end{rem}

\begin{rem} If $\varphi^{GD}$ is a finite potent G-Drazin inverse of a finite potent endomorphism $\varphi$ such that $W_{\varphi^{GD}} = W_\varphi$, then $$\tr_V \varphi^{GD} = \tr_V \varphi^D \text{ and } \Det^k_V \varphi^{GD} = \Det^k_V \varphi^D\, .$$

Indeed, in this case, if $\{\lambda_1, \dots, \lambda_n\}$ are the eigenvalues of $\varphi_{\vert_{W_\varphi}}$ in the algebraic closure of $k$ (with their multiplicity), one has that:

\begin{itemize}
\item  $\tr_V (\varphi^{GD}) = \lambda_1^{-1} + \dots + \lambda_n^{-1}$;

\medskip

\item $\Det^k_V (1 + \varphi^{GD}) = \prod_{i=1}^n (1 + \lambda_i^{-1})$.
\end{itemize}

\end{rem}

\section{Explicit computation of the G-Drazin inverses of a square matrix} \label{s:G-Dra-matri-3455}

The final section of this work is devoted to offer a method for computing explicitly all the G-Drazin inverses of a square matrix.

If $k$ is an arbitrary ground field, let us consider a square matrix $A \in {\text{Mat}}_{n\times n} (k)$ with $i(A) = r$.

Fixing a $k$-vector space $E$ with dimension n, a basis $B = \{e_1, \dots, e_n\}$ of $E$ and an endomorphism $\varphi \in \ed_k (E)$ associated with $A$ in the basis $B$, from the AST-decomposition $E = W_\varphi \oplus U_\varphi$ one has that \begin{equation} \label{eq:AST-demc-matr} A = P\cdot \begin{pmatrix} A_W & 0 \\ 0 & J_U \end{pmatrix} \cdot P^{-1}\, ,\end{equation} \noindent where $\varphi_{\vert_{W_\varphi}} \equiv A_W$, $J_U$ is the Jordan matrix determined by $\varphi_{\vert_{U_\varphi}}$ and $P$ is the corresponding base change matrix.

If $A \in {\text{Mat}}_{n\times n} (k)$ is again a square matrix with $i(A) = r$ and $\rk (A^i)$ is the rank of $A^i$, we can  write $\nu_i  (A) = n - \rk (A^i)$ for all $i\in \{1, \dots, r\}$ and we can consider the non-negative integers $\{\delta_1 (A), \dots, \delta_r (A)\}$ defined from the equations:  $$\begin{aligned} \delta_r (A)&= \nu_r (A) - \nu_{r-1} (A) \\ 2 \delta_r (A) + \delta_{r-1} (A) &= \nu_r (A) - \nu_{r-2} (A) \\ &\, \, \, \vdots \\ (r-1) \delta_r (A) + \dots + 2\delta_3 (A) + \delta_2 (A) &= \nu_r (A) - \nu_1 (A) \\ r\delta_r (A) + (r-1)\delta_{r-1} (A) + \dots + 2 \delta_2 (A) + \delta_1 (A) &= \nu_r (A)  \end{aligned} \, \quad  .$$

From these relations it is clear that $$\nu_1 (A) = \sum_{i = 1}^r \delta_i (A)  \text{ and }\sum_{j=1}^r \delta_j(A)  (\nu_r(A)  - j) = [\nu_1(A) - 1] \nu_r (A) \, .$$

Accordingly, the explicit expression of the matrix $J_U$ is $$J_U = \begin{pmatrix} A_1^1 & 0 & \dots & \dots & \dots & 0 & 0 \\ 0 & \ddots   & \ddots & \dots & \dots & \vdots & 0  \\ \vdots & \ddots  & A_1^{\delta_1 (A)} & \ddots & \dots & \vdots & \vdots \\ \vdots & \dots & \ddots & \ddots & \ddots & \dots & \vdots \\ \vdots & \dots & \dots & \ddots & A_{r}^1 & \ddots & \vdots \\ 0 &  \dots & \dots & \dots &  \ddots & \ddots & 0 \\ 0 & 0 & \dots & \dots & \dots & 0 & A_{r}^{\delta_{r} (A)} \end{pmatrix} \in {\text{Mat}}_{\nu_r\times \nu_r} (k)\,  ,$$\noindent where $$A_j^s = \begin{pmatrix} 0 & 0 & \dots &\dots & \dots & 0 \\ 1 & 0 & \ddots &\dots & \dots & 0 \\ 0 & 1 & 0 & \ddots & \dots & \vdots \\ \vdots & \ddots &  \ddots \begin{matrix} {\text{\tiny {j-1)}}} \\ \quad \end{matrix} & \ddots & \ddots & \vdots \\0 &  0 & \dots & 1 & 0 & 0  \\ 0 &  0 & \dots & 0 & 1 & 0  \end{pmatrix} \in {\text{Mat}}_{j\times j} (k)$$\noindent for every $j\in \{1, \dots, r\}$ and $1 \leq s \leq \delta_j (A)$.

With the above notation, if $A\{GD\}$ is the set of the G-Drazin inverses of $A$, it follows from Lemma \ref{l:char-gen-inv-hom-n2} that there exists a bijection \begin{equation} \label{eq:bij-agd-k} \begin{aligned}\quad  k^{\nu_1 (A) \cdot \nu_r (A)} \times  k^{[\nu_r (A) - \nu_1 (A)]} \times k^{[\nu_1(A) - 1] [\nu_r (A) - \nu_1 (A)]} &\overset {\sim} \longrightarrow A\{GD\} \\\big ( ((\alpha_{j,h}^s), (\alpha_{j,j',z}^s)), (\lambda_{j,t}^s), (\lambda_{j,j',x}^s) \big ) &\longmapsto  (A^{GD})_{ ((\alpha_{j,h}^s), (\lambda_{j,t}^s))}^{((\alpha_{j,j',z}^s), (\lambda_{j,j',x}^s)} \, ,\end{aligned}\end{equation} \noindent where $j,j'\in \{1,\dots, r\}$; $j\ne j'$; $1\leq h\leq j$; $z\in \{1, \dots, j\}$; ${{t\in \{2,\dots,j\}}}$; $x\in\{2, \dots, j'\}$;  $s\in \{1,\dots, \delta_j (A)\}$ and $$(A^{GD})_{ ((\alpha_{j,h}^s), (\lambda_{j,t}^s))}^{((\alpha_{j,j',z}^s), (\lambda_{j,j',x}^s)} = P \cdot \begin{pmatrix} (A_W)^{-1} & 0 \\ 0 & (J_U^-)_{ ((\alpha_{j,h}^s), (\lambda_{j,t}^s))}^{((\alpha_{j,j',z}^s), (\lambda_{j,j',x}^s)}  \end{pmatrix} \cdot P^{-1}$$\noindent with $$(J_U^-)_{ ((\alpha_{j,h}^s), (\lambda_{j,t}^s))}^{((\alpha_{j,j',z}^s), (\lambda_{j,j',x}^s)} = \big ( \big ((J_U^-)_{ ((\alpha_{j,h}^s), (\lambda_{j,t}^s))}^{((\alpha_{j,j',z}^s), (\lambda_{j,j',x}^s)} \big )_{lm} \big )_{1\leq l,m \leq \nu_1} \in {\text{Mat}}_{\nu_r\times \nu_r} (k)$$ such that 

\begin{itemize}

\item if $j\in \{1, \dots, r\}$ and $1\leq s \leq \delta_j (A)$ are such that $l = (\sum_{i=1}^{j-1} \delta_i) + s$, then $(J_U^-)_{ll} \in {\text{Mat}}_{j\times j} (k)$ with $$\begin{aligned} \big ((J_U^-)_{ ((\alpha_{j,h}^s), (\lambda_{j,t}^s))}^{((\alpha_{j,j',z}^s), (\lambda_{j,j',x}^s)}\big )_{ll} &=  \big (A_{\{\alpha_{j,1}^{s}, \dots, \alpha_{j,j-1}^{s}\}}^{\{\lambda_{j,1}^{s}, \dots, \lambda_{j,j}^{s}\}}\big )^- = \\ &=
\begin{pmatrix} \alpha_{j,1}^{s} & 1 & 0 & \dots & \dots & 0  \\ \alpha_{j,2}^{s} &  0  & 1 & 0 & \ddots & \vdots \\ \vdots & \vdots & \ddots & \ddots & \ddots & \vdots \\ \alpha_{j,j-2}^{s} & 0 & \dots & \ddots & 1 & 0 \\  \alpha_{j,j-1}^{s} & 0 & \dots & \dots & 0 & 1 \\ \alpha_{j,j}^{s} & \lambda_{j,2}^s &  0 & \dots &  \lambda_{j,j-1}^s &  \lambda_{j,j}^s \end{pmatrix}\, ,\end{aligned}$$\noindent for all $\alpha_{j,h}^s, \lambda_{j,t}^s\in k$, $h\in \{1, \dots, j\}$, ${{t\in \{2,\dots,j\}}}$, $j\in \{1,\dots, r\}$ and $s\in \{1,\dots, \delta_j (A)\}$;

\item  if $j,j' \in \{1, \dots, r\}$, $1\leq s \leq \delta_j (A)$ and $1\leq s' \leq \delta_{j'} (A)$ are such that  $$l = (\sum_{i=1}^{j-1} \delta_i) + s  \text{ and } m = (\sum_{i=1}^{j'-1} \delta_i) + s'\, ,$$\noindent with $l\ne m$,  then $(J_U^-)_{lm} \in {\text{Mat}}_{j\times j'} (k)$ where $$\begin{aligned} \big ((J_U^-)_{ ((\alpha_{j,h}^s), (\lambda_{j,t}^s))}^{((\alpha_{j,j',z}^s), (\lambda_{j,j',x}^s)} \big)_{lm}  &= \big (A_{\{\alpha_{j,j',1}^{s}, \dots, \alpha_{j,j',j}^{s}\}}^{\{\lambda_{j,j',2}^{s}, \dots, \lambda_{j,j',j'}^{s}\}} \big )^- = \\ &= \begin{pmatrix} \alpha_{j,j',1}^s & 0 & \dots & \dots & 0 \\ \alpha_{j,j',2}^s & \vdots & \dots & \dots & \vdots \\ \vdots & \vdots & \dots & \dots & \vdots \\ \alpha_{j,j',j-1}^s &  0 & \dots & \dots & 0 \\ \alpha_{j,j',j}^s & \lambda_{j,j',2}^s & \dots & \dots & \lambda_{j,j',j'}^s \end{pmatrix} \end{aligned}$$\noindent for every $\alpha_{j,j',z}^s, \lambda_{j,j',x}^s \in k$, $j\ne j'$, $1\leq j,j' \leq r$, $z\in\{1, \dots, j\}$, \linebreak $x\in\{2, \dots, j'\}$ and $s\in \{1, \dots, \delta_j (A)\}$.

\end{itemize}  

If ${\tilde A} \in {\text{Mat}}_{n\times n} (k)$ with $i ({\tilde A}) = 1$, it follows from (\ref{eq:bij-agd-k}) that there exists a bijection ${\tilde A}\{GD\} \overset {\sim} \longrightarrow k^{[\nu_1 ({\tilde A})^2]}$.

From the results of this work, we can finally offer the following algorithm for computing the G-Drazin inverses of $A \in {\text{Mat}}_{n\times n} (k)$.

\begin{enumerate}

\item Fix a $k$-vector space $E$ with dimension n, a basis $B = \{e_1, \dots, e_n\}$ of $E$ and an endomorphism $\varphi \in \ed_k (E)$ associated with $A$ in the basis $B$, to facilitate the computations.

\item Compute the AST-decomposition $E = W_\varphi \oplus U_\varphi$ and the matrix expression (\ref{eq:AST-demc-matr}) for $A$.

\item Calculate the non-negative integer numbers $\{\nu_1 (A), \dots, \nu_r (A) \}$ and\linebreak  $\{\delta_1(A), \dots, \delta_r (A)\}$.

\item Construct the matrices $(J_U^-)_{ ((\alpha_{j,h}^s), (\lambda_{j,t}^s))}^{((\alpha_{j,j',z}^s), (\lambda_{j,j',x}^s)}$ and compute $(A_W)^{-1} $.

\item Get all the G-Drazin inverses $(A^{GD})_{ ((\alpha_{j,h}^s), (\lambda_{j,t}^s))}^{((\alpha_{j,j',z}^s), (\lambda_{j,j',x}^s)}$ of $A$.

\end{enumerate}

\begin{rem} We wish remark that is not necessary to compute the characteristic polynomial $c_A (x)$ in the method offered in this paper for calculate all the G-Drazin inverses of a square matrix $A$ with $i(A) = r$, because we can obtain the matrices $A_W$ y $A_U$ by computing $R(A^r)$ and $N(A^r)$, where $R(B)$ and $N(B)$ are the range and the nullspace of a matrix $B$ respectively.
\end{rem}

\begin{exam} \label{ex:g-dra-finmatrz} Let us consider an arbitrary field $k$ and the matrix $$A = \begin{pmatrix} -9 & -7 & 11 & -3 & -6 & -4 & -2 \\ -3 & 1 & 2 & 1 & 1 & 0 & 1 \\ -13 & -8 & 15 & -3 & -7 & -5 & -2 \\ -4 & -3 & 5 & -1 & -3 & -2 & -1 \\ -13 & -12 & 17 & -6 & -11 & -7 & -4 \\ 11 & 10 & -14 & 5 & 9 & 6 & 3 \\ 8 & 6 & -10 & 3 & 6 & 4 & 2 \end{pmatrix} \in {\text{Mat}}_{7\times 7} (k)\, .$$

We shall compute all the G-Drazin inverses $A^{GD}$ of $A$.

Let us now fix a $k$ vector space $E$ with basis $\{e_1, e_2, e_3, e_4, e_5, e_6, e_7\}$ and an endomorphism $\varphi \in \ed_k (E)$ such that $\varphi \equiv A$ in this basis.

It is easy to check that $i(A) = 3$, $\rk (A) = 5$, $\rk (A^2) = 3$ and $\rk (A^3) = 2$. Accordingly, $\nu_1 (A)  = 2$, $\nu_2(A) = 4$, $\nu_3 (A) = 5$, $\delta_1 (A) = 0$, $\delta_2 (A) = 1$ and $\delta_3 (A) = 1$.

Now, a non-difficult computation shows that $W_\varphi = \langle e_1 + e_2 + 2e_3 + e_5 - e_7, e_4 - e_6\rangle$, $U_\varphi = \langle e_1 + e_3, e_2 - e_5, -e_5 + e_6, e_4 - e_7, e_1 + e_3 - e_7 \rangle$, $$A_W = \begin{pmatrix} 2 & 1 \\ 1 & 1 \end{pmatrix}$$\noindent and $$A_U = \begin{pmatrix} 3 & -1 & 3 & -2 & 0 \\ -1 & 0 & -1 & 0 & 0 \\ -3 & 1 & -3 & 2 & 0 \\ 1 & 0 & 1 & 0 & 0 \\ -1 & 0 & -1 & 1 & 0 \end{pmatrix}\, .$$

Moreover, from the Jordan basis $$\begin{aligned} \{&-e_1 + e_2 - e_3 - e_5 - e_7, - e_1 - e_3 - e_5 + e_6, e_1 + e_2 + e_3 + e_4 - e_5 - e_7, \\ &-e_2 + e_4 + e_5 - e_7, - e_5 + e_6 + e_7\}\end{aligned}$$ of $U_\varphi$ induced by $\varphi_{\vert_{U_\varphi}}$, one gets that $$A = P \cdot \begin{pmatrix} 2 & 1 & 0 & 0 & 0 & 0 & 0 \\ 1 & 1 & 0 & 0 & 0 & 0 & 0 \\   0 & 0 & 0 & 0 & 0 & 0 & 0 \\ 0 & 0 & 1 &  0 & 0 & 0 & 0 \\  0 & 0 & 0 & 0 & 0 & 0 & 0 \\  0 & 0 & 0 & 0 & 1 & 0 & 0 \\ 0 & 0 & 0 & 0 & 0 & 1 & 0 \end{pmatrix} \cdot P^{-1}\, ,$$\noindent with $$P = \begin{pmatrix} 1 & 0 & -1 & -1 & 1 & 0 & 0 \\ 1 & 0 & 1 & 0 & 1 & -1 & 0 \\ 2 & 0 & -1 & -1 & 1 & 0 & 0 \\ 0 & 1 & 0 & 0  & 1 & 1 & 0 \\ 1 & 0 & -1 & -1 & -1 & 1 & -1 \\ 0 & -1 & 0 & 1 & 0 & 0 & 1 \\ -1 & 0 & -1 & 0 & -1 & -1 & 1\end{pmatrix}$$\noindent and $$P^{-1} = \begin{pmatrix} -1 & 0 & 1 & 0 & 0 & 0 & 0 \\ -2 & -1 & 2 & 0 & -1 & -1 & 0 \\ 5 & 6 & -7 & 3 & 5 & 3 & 2 \\ -8 & -8 & 10 & -4 & -7 & -4 & -3 \\ -1 & -2 & 2 & -1 & -2 & -1 & -1 \\ 3 & 3 & -4 & 2 & 3 & 2 & 1 \\ 6 & 7 & -8 & 4 & 6 & 4 & 3 \end{pmatrix}\, .$$

Thus, it follows from the method described above that a G-Drazin inverse of $A$ has the explicit expression $$(A^{GD})_{ ((\alpha_{j,h}^1), (\lambda_{j,t}^1))}^{(\gamma_{j,j',z}^1)} = P \cdot \begin{pmatrix} 1 & -1 & 0 & 0 & 0 & 0 & 0 \\ -1 & 2 & 0 & 0 & 0 & 0 & 0 \\   0 & 0 & \alpha_{2,1}^1 & 1 & \alpha_{2,3,1}^1 & 0 & 0 \\ 0 & 0 & \alpha_{2,2}^1 &  \lambda_{2,2}^1 &  \alpha_{2,3,2}^1 &  \lambda_{2,3,2}^1 &  \lambda_{2,3,3}^1 \\  0 & 0 &  \alpha_{3,2,1}^1 & 0 & \alpha_{3,1}^1 & 1 & 0 \\  0 & 0 &  \alpha_{3,2,2}^1 & 0 & \alpha_{3,2}^1 & 0 & 1 \\ 0 & 0 &  \alpha_{3,2,3}^1 & \lambda_{3,2,2}^1 & \alpha_{3,3}^1 & \lambda_{3,2}^1 & \lambda_{3,3}^1 \end{pmatrix} \cdot P^{-1}\, ,$$\noindent with $\alpha_{j,h}^1, \lambda_{j,t}^1, \alpha_{j,j',z}^1,  \lambda_{j,j',x}^1\in k$ for every $j\ne j'$, $2\leq j,j' \leq 3$, $h\in \{1, \dots, j\}$, ${{t\in \{2,\dots,j\}}}$, $z\in\{1, \dots, j\}$ and $x\in \{2, \dots, j'\}$.

Hence, with the data of this example, we have the following bijection that determines all the G-Drazin inverses of $A$: $$\begin{aligned}  k^{10} \times k^3 \times k^3  &\overset {\sim} \longrightarrow A\{GD\}  \\   \big ( ((\alpha_{j,h}^1), ( \alpha_{j,j',z}^1))_{\{j,j',h,z\}},  (\lambda_{j,t}^1)_{\{j,t\}}, ( \lambda_{j,j',x}^1)_{\{j,j',x\}} \big ) &\longmapsto (A^{GD})_{((\alpha_{j,h}^1), (\lambda_{j,t}^1))}^{(\gamma_{j,j',z}^1)}\, ,\end{aligned}$$\noindent with  $j\ne j'$, $2\leq j,j' \leq 3$, $h\in \{1, \dots, j\}$, ${{t\in \{2,\dots,j\}}}$, $z\in\{1, \dots, j\}$ and $x\in \{2, \dots, j'\}$.
\end{exam}

Finally, we shall study the relationships that there exist between G-Drazin inverses and the core-nilpotent decomposition of a matrix $A$.

\begin{rem} If $k$ is an arbitrary ground field, $A \in {\text{Mat}}_{n\times n} (k)$, $A = A_1 + A_2$ is its core-nilpotent decomposition and $A^{GD}$ is a G-Drazin inverse of $A$ with core-nilpotent decomposition $A^{GD} = (A^{GD})_1 + (A^{GD})_2$,  Example \ref{ex:g-dra-finmatrz} shows that, in \linebreak general, $(A^{GD})_1$ is not a G-Drazin inverse of $A_1$ and $(A^{GD})_2$ is not a G-Drazin inverse of $A_2$. 

	As a counterexample of this fact we offer the following: if $A$ is the matrix studied in Example \ref{ex:g-dra-finmatrz}, an easy computation shows that its core-nilpotent decomposition is $A = A_1 + A_2$ with $$A_1 = \begin{pmatrix} -4 & -1 & 4 & 0 & -1 & -1 & 0 \\   -4 & -1 & 4 & 0 & -1 & -1 & 0 \\ -8 & -2 & 8 & 0 & -2 & -2 & 0 \\ -3 & -1 & 3 & 0 & -1 & -1 & 0 \\ -4 & -1 & 4 & 0 & -1 & -1 & 0 \\ 3 & 1 & -3 & 0 & 1 & 1 & 0 \\ 4 & 1 & -4 & 0 & 1 & 1 & 0 \end{pmatrix}$$ and $$A_2 = \begin{pmatrix} -5 & -6 & 7  & -3 & -5 & -3 & -2 \\ 1 & 2 & -2 & 1 & 2 & 1 & 1 \\ -5 & -6 & 7  & -3 & -5 & -3 & -2 \\ -1 & -2 & 2 & -1 & -2 & -1 & -1 \\  -9 & -11 & 13 & -6 & -10 & -6 & -4 \\ 8 & 9 & -11 & 5 & 8 & 5 & 3 \\ 4 & 5 & -6 & 3 & 5 & 3 & 2 \end{pmatrix}\, .$$

If we now consider $$A^{GD} = \begin{pmatrix} 7 & 6 & -8 & 3 & 6 & 4 & 2 \\ -10 & -11  & 13 & -6 & -9 & -5 & -5 \\ 8 & 7 & -9 & 3 & 7 & 5 & 2 \\ 6 & 8 & -9 & 6 & 7 & 4 & 4 \\ 8 & 9 & -10 & 4 & 8 & 5 & 4 \\ 7 & 6 & -8 & 2 & 5 & 4 & 1 \\ -3 & -5 & 5 & -3 & -5 & -4 & 2  \end{pmatrix} \, ,$$\noindent which is the G-Drazin inverse of $A$ determined by $\lambda_{2,1}^1 = \lambda_{3,1}^1 = 1$ and  otherwise $\alpha_{j,h}^1 = \lambda_{j,t}^1 = \gamma_{j,j',z}^1 = 0$, one has that $(A^{GD})_1 = (A^{GD})^{-1}$ and $(A^{GD})_2 = 0$, and we can  immediately check that $(A^{GD})_1$ is not a G-Drazin inverse of $A_1$ and $(A^{GD})_2$ is not a G-Drazin inverse of $A_2$.

	Furthermore, if $(A_1)^{GD}$ is a G-Drazin inverse of $A_1$ and $(A_2)^{GD}$ is a G-Drazin inverse of $A_2$, in general, one has that ${\bar A}^{GD} = (A_1)^{GD}+ (A_2)^{GD}$ is not a G-Drazin inverse of $A$, as can be deduced from this counterexample: keeping again the data of Example \ref{ex:g-dra-finmatrz}, if we consider $$(A_1)^{GD} =  P \cdot \begin{pmatrix} 1 & -1 & 0 & 0 & 0 & 0 & 0 \\ -1 & 2 & 0 & 0 & 0 & 0 & 0 \\   0 & 0 & 0 & 0 & 0 & 1 & 0 \\ 0 & 0 & 0 &  0 & 0 & 0 & 0 \\  0 & 0 & 0 & 0 & 0 & 0 & 0 \\  0 & 0 & 0 & 0 & 0 & 0 & 0 \\ 0 & 0 & 0 & 0 & 0 & 0 & 0  \end{pmatrix} \cdot P^{-1}$$ \noindent and $$(A_2)^{GD} =  P \cdot \begin{pmatrix} 0 & 0 & 0 & 0 & 0 & 0 & 0 \\ 0 & 0 & 0 & 0 & 0 & 0 & 0 \\   0 & 0 & 0 & 1 & 0 & 0 & 0 \\ 0 & 0 & 0 &  0 & 0 & 0 & 0 \\  0 & 0 & 0 & 0 & 0 & 1 & 0 \\  0 & 0 & 0 & 0 & 0 & 0 & 1 \\ 0 & 0 & 0 & 0 & 0 & 0 & 0  \end{pmatrix} \cdot P^{-1}\, ,$$ then it is clear that $${\bar A}^{GD} =  P \cdot \begin{pmatrix} 1 & -1 & 0 & 0 & 0 & 0 & 0 \\ -1 & 2 & 0 & 0 & 0 & 0 & 0 \\   0 & 0 & 0 & 1 & 0 & 1 & 0 \\ 0 & 0 & 0 &  0 & 0 & 0 & 0 \\  0 & 0 & 0 & 0 & 0 & 1 & 0 \\  0 & 0 & 0 & 0 & 0 & 0 & 1 \\ 0 & 0 & 0 & 0 & 0 & 0 & 0  \end{pmatrix} \cdot P^{-1}$$\noindent is not a G-Drazin inverse of $A$.

\end{rem}

\end{document}